\documentclass[12pt,leqno]{article}
\usepackage{amsfonts}
\usepackage{amsmath, amsthm, amsfonts, amssymb, color}
\usepackage{mathrsfs}
\usepackage{color}
\newtheorem{thm}{Theorem}[section]

\newtheorem{lem}[thm]{Lemma}

\newtheorem{rem}[thm]{Remark}
\theoremstyle{definition}

\newcommand{\scr}[1]{\mathscr #1}
\definecolor{wco}{rgb}{0.5,0.2,0.3}

\numberwithin{equation}{section} \theoremstyle{remark}

\newcommand{\ua}{\uparrow}

\title{{\bf    Approximation of SPDEs with H\"older Continuous Drifts} }
\author{
{\bf  Jianhai Bao$^{b,c)}$,  Xing Huang$^{a)}$, Chenggui Yuan$^{c)}$}\\
\footnotesize{$^{a)}$Center of Applied Mathematics, Tianjin
University, Tianjin 300072, China}\\
\footnotesize{$^{b)}$School of Mathematics and Statistics, Central
South
University, Changsha 410083, China}\\
\footnotesize{$^{c)}$Department of Mathematics, Swansea University,
Singleton Park, SA2 8PP, UK}}
\begin{document}
\allowdisplaybreaks
\def\R{\mathbb R}  \def\ff{\frac} \def\ss{\sqrt} \def\B{\mathbf
B}
\def\N{\mathbb N} \def\kk{\kappa} \def\m{{\bf m}}
\def\ee{\varepsilon}\def\ddd{D^*}
\def\dd{\delta} \def\DD{\Delta} \def\vv{\varepsilon} \def\rr{\rho}
\def\<{\langle} \def\>{\rangle} \def\GG{\Gamma} \def\gg{\gamma}
  \def\nn{\nabla} \def\pp{\partial} \def\E{\mathbb E}
\def\d{\text{\rm{d}}} \def\bb{\beta} \def\aa{\alpha} \def\D{\scr D}
  \def\si{\sigma} \def\ess{\text{\rm{ess}}}
\def\beg{\begin} \def\beq{\begin{equation}}  \def\F{\scr F}
\def\Ric{\text{\rm{Ric}}} \def\Hess{\text{\rm{Hess}}}
\def\e{\text{\rm{e}}} \def\ua{\underline a} \def\OO{\Omega}  \def\oo{\omega}
 \def\tt{\tilde} \def\Ric{\text{\rm{Ric}}}
\def\cut{\text{\rm{cut}}} \def\P{\mathbb P} \def\ifn{I_n(f^{\bigotimes n})}
\def\C{\scr C}   \def\G{\scr G}   \def\aaa{\mathbf{r}}     \def\r{r}
\def\gap{\text{\rm{gap}}} \def\prr{\pi_{{\bf m},\varrho}}  \def\r{\mathbf r}
\def\Z{\mathbb Z} \def\vrr{\varrho} \def\ll{\lambda}
\def\L{\scr L}\def\Tt{\tt} \def\TT{\tt}\def\II{\mathbb I}
\def\i{{\rm in}}\def\Sect{{\rm Sect}}  \def\H{\mathbb H}
\def\M{\scr M}\def\Q{\mathbb Q} \def\texto{\text{o}} \def\LL{\Lambda}
\def\Rank{{\rm Rank}} \def\B{\scr B} \def\i{{\rm i}} \def\HR{\hat{\R}^d}
\def\to{\rightarrow}\def\l{\ell}\def\iint{\int}
\def\EE{\scr E}\def\no{\nonumber}
\def\A{\scr A}\def\V{\mathbb V}\def\osc{{\rm osc}}
\def\BB{\scr B}\def\Ent{{\rm Ent}}
\def\U{\scr U}\def\lf{\lfloor}
\def\rf{\rfloor}\def\1{\lesssim}\def\8{\infty}

\maketitle

\begin{abstract}
In this paper, exploiting the regularities of the corresponding
Kolmogorov equations involved we investigate strong convergence of
exponential integrator scheme for a range of stochastic partial
differential equations, in which the drift term is H\"older
continuous, and reveal the rate of convergence.

\end{abstract} \noindent
 \vskip 2cm

\section{Introduction}
The numerical approximation of stochastic partial differential
equations (SPDEs) has been a very active field of research. Due to
the infinite dimensional nature of the driving noise processes, in
order to be able to simulate  a numerical approximation on a
computer, both temporal discretization and spatial discretization
are  implemented. The temporal discretization is achieved generally
by Euler type approximations,  Milstein type approximations, and
splitting-up method (see, e.g.,
\cite{BL,D11,DJM,G03,H03,JR15,LCP,S99}), and the spatial
discretization
 is in general done by finite element, finite difference and spectral
Galerkin methods (see, e.g., \cite{J11,K}). Also, there are some
alternative schemes (for example, stochastic Taylor expansion
\cite{JK11}) to approximate SPDEs.  In contrast to substantial
literature on approximations of SPDEs with regular coefficients, the
counterpart for SPDEs with irregular terms (e.g., H\"older
continuous drifts) is scarce. Whereas, our goal in this paper is to
make an attempt to discuss strong convergence of an exponential
integrator (EI) scheme, coupled with a Galerkin scheme for the
spatial discretization (see \eqref{w4} and \eqref{eq6} below), for a
class of SPDEs with H\"older continuous drifts. With regard to
convergence of EI scheme for SPDEs with smooth drift coefficients,
we refer to \cite{JK,KLNS,GR,GT} for further details, to name a few.
Also, there is a number of literature on approximation of SPDEs with
non-globally Lipschitz continuous nonlinearities; see, for instance,
  \cite{GS,JP}.

Let
$(\mathbb{H},\langle\cdot,\cdot\rangle_\mathbb{H},\|\cdot\|_\mathbb{H})$
be  a real separable Hilbert space. Denote by $\mathscr{L}(\H)$
(resp. $\mathscr{L}_2(\H)$)
 the space  of all bounded linear operators (resp. Hilbert-Schmidt operators) on $\H$.  Let $\|\cdot\|$ (resp.
$\|\cdot\|_{\mathscr{L}_2}$) stand for the operator
 norm (resp. the Hilbert-Schmidt norm).
Let $(W_t)_{t\ge0}$ be an $\H$-valued cylindrical Wiener process
defined on a complete probability space
$(\OO,\F,(\F_t)_{t\ge0},\P)$, i.e.,
$W_t=\sum_{k=1}^{\infty}{\bb_t^{(k)}e_k},$ where
$(\bb^{(k)})_{k\ge1}$ is a sequence of independent real-valued
Brownian motions  on  the probability space $(\OO, \F,
(\F_t)_{t\ge0}, \P)$ and $(e_k)_{k\ge1}$ is an orthonormal basis of
$\H$. Fix $T>0$ and   set $
\|f\|_{T,\infty}:=\sup_{t\in[0,T],x\in\H}\|f(t,x)\|$ for an
operator-valued map $f$ on $[0,T]\times\H$. Let $\B_b(\H;\H)$ be the
collection of all bounded measurable functions $f:\H\rightarrow\H.$

%
%

Consider the following semi-linear SPDE on $\mathbb{H}$
\beq\label{1.1} \d X_{t}= \{A X_{t}+b_t(X_{t})\}\d t+\d
W_t,~~~t\in[0,T],~~~X_0=x\in\H,
\end{equation}
where
\begin{enumerate}
\item[{\bf (A1)}] $(-A,\D(-A))$ is a positive definite self-adjoint operator on $\mathbb{H}$
having discrete spectrum with all eigenvalues
$(0<)\ll_1\le\ll_2\le\cdots$ counting multiplicities such that
\begin{equation}\label{eq1}
\sum_{i=1}^\8\ff{1}{\ll_i^{1-\aa}}<\8
\end{equation}
for some $\aa\in(0,1).$
\item[{\bf (A2)}]
$b: [0,T]\times \mathbb{H}\rightarrow\mathbb{H}$  is uniformly
bounded (i.e., $\|b\|_{T,\infty}<\8$) and there exist $c,\bb>0$ and
$\vv\in(0,1)$ with $\ff{2\bb}{2-\vv}\ge1-\aa$ such that
\beq\label{1.3} \|b_t(x)-b_t(y)\|_\H\leq
c\,\sum_{i=1}^{\8}\ff{1}{\ll_i^{\bb}}|x_i-y_i|^\vv,\ \ t\in[0,T],
~~x,y\in \mathbb{H},
 \end{equation}
where $x_i:=\<x,e_i\>$ and $y_i:=\<y,e_i\>$.

\item[{\bf (A3)}] For $c>0$ and $\vv\in(0,1)$ in \eqref{1.3},
\beq\label{1.4} \|b_{s}(x)-b_{t}(x)\|_\H\leq c|s-t|^\vv, \quad s,
t\in[0,T], x\in \mathbb{H}.
\end{equation}
\end{enumerate}

It is easy to see that \eqref{1.3} is equivalent to that $b_t$ is
continuous and that \beq\label{e3}
\|b_t(x)-b_t(x+\<y-x,e_i\>e_i)\|_\H\leq
\ff{c\,}{\ll_i^{\bb}}|x_i-y_i|^\vv,\ \ t\in[0,T], ~~x,y\in
\mathbb{H}.
 \end{equation}
By the H\"older inequality, along with \eqref{eq1} and
$\ff{2\bb}{2-\vv}\ge1-\aa$, \eqref{1.3} implies that \beq\label{e2}
\|b_t(x)-b_t(y)\|_\H\leq c_0\|x-y\|^\vv_\H,\ \ t\in[0,T], ~~x,y\in
\mathbb{H}
 \end{equation}
for some $c_0>0.$ That is, $b_t$ is H\"older continuous of order
$\vv$ w.r.t. the spatial variable. Thus, according to \cite[Theorem
1.1]{W}, \eqref{1.1} has a unique mild solution, i.e., there exists
a unique continuous adapted process $(X_t)_{t\ge0}$ such that
$\P$-a.s.,
\begin{equation*}
X_t=\e^{tA}x+\int_0^t\e^{(t-s)A}b_s(X_s)\d s+\int_0^t\e^{(t-s)A}\d
W_s,~~~~t>0.
\end{equation*}

  For any $n\in\mathbb{N}$, let $\pi_n:\H\mapsto
\H_n:=\mbox{span}\{e_1,\cdots,e_n\}$ be the orthogonal projection,
$A_n=\pi_nA$, $b^{(n)}_t=\pi_nb_t$ and $W_t^{(n)}=\pi_nW_t$. 
With the notation above, we consider the following
finite-dimensional approximation associated with \eqref{1.1} on
$\H_n\simeq\R^n$
\begin{equation}\label{f1}
\d X^{(n)}_t=\{A_nX^{n}_t+b^{(n)}_t(X^{(n)}_t)\}\d t+\d
W_t^{(n)},~~t>0,~~ X^{(n)}_0=x_n:=\pi_nx,
\end{equation}
which is the Galerkin projection of \eqref{1.1} onto $\H_n$. Since
$A_nx=Ax$ for any $x\in\H_n$ and $b_n$ is H\"older continuous in
terms of \eqref{e3}, by virtue of \cite[Theorem 1.1]{W}, \eqref{f1}
has a unique strong solution.

Now we define   numerical schemes to approximate $X^{(n)}_t$ in
time, which are called discrete-time EI scheme: for a stepsize
$\dd\in(0,1)$ and each integer $k\ge0$,
\begin{equation}\label{w4}
\bar{Y}^{(n),\dd}_{(k+1)\dd}=\e^{\dd
A_n}\{\bar{Y}^{(n),\dd}_{k\dd}+b^{(n)}_{k\dd}(\bar{Y}^{(n),\dd}_{k\dd})\dd+\bigtriangleup
W_k^{(n)}\}, ~~\bar{Y}^{(n),\dd}_0=x_n,
\end{equation}
which is also named as Lord-Rougemont scheme (see, e.g.,
\cite[(4.5)]{JK} and \cite[(3.2)]{GR}), where $\bigtriangleup
W_k^{(n)}:=W_{(k+1)\dd}^{(n)}-W_{k\dd}^{(n)}$, and continuous-time
EI scheme
\begin{equation}\label{eq6}
\begin{split}
Y^{(n),\dd}_t&=\e^{tA_n}x_n+\int_0^t\e^{(t-s_\dd)A_n}b_{s_\dd}^{(n)}(Y^{(n),\dd}_{s_\dd})\d
s+\int_0^t\e^{(t-s_\dd)A_n}\d W^{(n)}_s,~~~~t\ge0,
\end{split}
\end{equation}
 where $t_\dd:=\lf t/\dd\rf\dd$ with  $\lf t/\dd\rf$ being  the
integer part of $t/\dd$. It is easy to see that
$Y^{(n),\dd}_{k\dd}=\bar{Y}^{(n),\dd}_{k\dd}$ for any $k\ge0.$

The main result of this paper is stated as follows.
\beg{thm}\label{T1.1} {\rm Assume that  {\bf (A1)}-{\bf (A3)} hold,
and suppose $\ff{1}{2}>\nu:=\ff{\vv+2\bb\wedge\aa\vv^2}{2}+\aa-1>0$
and the initial value  $x\in\D(A)$. Then,  there exists some
$C=C(\nu,T)>0 $ such that
\begin{equation}\label{eq12}
\int_0^T\E\|X_{t}-Y^{(n),\dd}_t\|^2_\H\d t\le
C\Big\{\dd^{\nu}+\ff{1}{\ll_n^{\nu}}\Big\}.
\end{equation}


 }
\end{thm}
\paragraph{Remark 1.1}
If $\bb\ge\ff{\aa\vv^2}{2}$, then $\nu>0$ reduces to
$\vv+\aa\vv^2>2(1-\aa)$ which is  equivalent to
$\aa>\ff{2-\vv}{2+\vv^2}$.  On the other hand, let
  $Ax=\pp^2_\xi x$ for $x\in\mathcal
{D}(A):=H^2(0,\pi)\cap H^1_0(0,\pi)$. Then $A$ is a self-adjoint
negative operator and $Ae_n=-n^2e_n,\ n\in\mathbb{N}$, where
$e_n(\xi)=(2/\pi)^{1/2}\sin n\xi$ for $\xi\in[0,\pi]$ and $
n\in\mathbb{N}$. So, $\ll_n=n^2$ and we can take some $\aa\in(0,1)$
such that \eqref{eq1} holds provided that $\vv>0.732$. Thus, taking
$\dd=\ff{T}{n}$ we obtain the convergence rate (in the sense of
\eqref{eq12}) for EI scheme of stochastic heat equations driven by
additive space-time white noise, where the drift $b$ satisfies
\eqref{1.3}.

\begin{rem}
{\rm To avoid complicated computation, in the present setup we work
only on the case that the drift is uniformly bounded. Nevertheless,
employing the standard cut-off approach (see, e.g., \cite{BHY}), we
of course can extend our framework to the setting that the drift
coefficient is unbounded.

}
\end{rem}

The remainder of this paper is organized as follows: In Section
\ref{sec2}, we investigate the regularities of the Kolmogrov
equation which enables us to construct a transformation, which is a
$C^2$-diffeomorphism and provides regular presentation of $X_t$ and
$Y^{(n),\dd}_t$, respectively; In Section \ref{sec3}, we fucus on
the spatial error and the temporal error so as to complete the proof
of Theorem \ref{T1.1}.

Convention: The letter $c$ or $C$ with or without subscripts will
denote an unimportant constant, whose values may change in different
places. Moreover, we use the shorthand notation $a\1b$ to mean $a\le
c\,b$. If the constant $c$ depends on a parameter $p$, we shall also
write $c_p$ and $a\1_p b$.

\section{Regularity of Kolmogorov Equation}\label{sec2}

The proof of Theorem \ref{T1.1} relies heavily on the regularity of
the following parabolic type partial differential equation,  for
fixed $T>0$, any $\lambda>0$ and $x\in\cup_{n=1}^\8\H_n$,
\beg{equation}\label{2.1}
\Big(\partial_tu_t^\ll+\nabla_{b_t}u_t^\ll+b_{t}+\frac{1}{2}\sum_{i=1}^{\infty}\nn^2_{e_i}u_t^\ll+\nabla
_{A\cdot}u_t^\ll\Big)(x)=\ll u_t^\ll(x),~~~t\in[0,T],~~~u_T^\ll=0,
\end{equation}
 where $\nn^2_{e_i}:=\nabla_{e_{i}}\nabla_{e_{i}}$, the second order gradient operator along the direction $e_i$. To characterize
the regular property of the solution to \eqref{2.1}, we consider the
following  O-U process \beg{equation}\label{O-U} \d Z_t^{x}=A
Z_t^{x}\d t+\d W_t, ~~~t\geq 0,  \quad Z_0^{x}=x.
\end{equation}
Under {\bf (A1)}, it is well known that \eqref{O-U} has an up to
modifications unique mild solution $(Z_{t}^{x})_{t\geq 0}$ (see,
e.g., \cite{GZ}) with the associated Markov semigroup
$(P_{t}^{0})_{t\ge0}$. Consider the following integral equation
\begin{equation}\label{PDE} u^\ll_t(x)=\int_t^T
\e^{-\ll(s-t)}P_{s-t}^{0}(\nn_{b_s}u^\ll_s+b_s)(x)\d s, \quad
x\in\H,~~ t\in[0,T].
\end{equation}
It is well known that, for any $x\in\cup_{n=1}^\8\H_n$, \eqref{2.1}
and \eqref{PDE} are equivalent mutually.


Before we start the proof of Theorem \ref{T1.1}, we prepare some
auxiliary lemmas.

\beg{lem}\label{L2.1} {\rm  Let {\bf (A1)} and {\bf (A2)} hold.
Then,

\beg{enumerate}
\item[(1)]
 There exists $\ll_T>0$ such that, for any $\ll\ge\ll_T$, \eqref{PDE} has a unique solution
$u^\lambda\in C([0,T]; C_b^2(\mathbb{H}; \mathbb{H}))$, which
 satisfies
 \beq\label{u1-0} \lim_{\lambda\to
\infty}\{\|u^\lambda\|_{T,\infty}+\|\nabla u^\lambda
(-A)^{\kappa}\|_{T,\infty}+\|\nabla^{2} u^\lambda
\|_{T,\infty}\}=0,~~~~\kappa\in[0,1/2).
\end{equation}

\item[(2)] For any $\theta\in[0,\aa)$ and $\ll\ge\ll_T,$
 \beg{equation}\label{w1}
 \sum_{i=1}^\8\ll_i^{\theta}\|\nn_{e_i}u^\ll\|_{T,\8}^2<\8.
\end{equation}

\end{enumerate}

}
\end{lem}

\begin{proof}
For  the wellposedness of \eqref{PDE} and  $\lim_{\lambda\to
\infty}\{\|u^\lambda\|_{T,\infty}+\|\nabla^{2} u^\lambda
\|_{T,\infty}\}=0$ in \eqref{u1-0}, we refer to \cite[Lemma 2.3]{W}
for more details. Hereinafter, we aim to show that
\begin{equation}\label{eq8}
\lim_{\ll\to \8}\|\nn u^\ll (-A)^\kk
\|_{T,\8}=0,~~~~~\kappa\in[0,1/2).
\end{equation}
 Note that
the following Bismut formula \beg{equation}\label{2.6} \nn_\eta
P_t^0f(x)=\mathbb{E}\Big(\frac{f(Z_t^x)}{t}\int_0^t\< \nn_\eta
Z_s^x,\d W_s\>\Big),\ \ t>0, x,\eta\in\H, f\in\B_b(\H;\H)
\end{equation}
holds; see, e.g., \cite[(2.8)]{W} by using the Mallivin calculus. By
H\"older's inequality and It\^o's isometry, together with $\nn_\eta
Z_t^x=\e^{tA}\eta$, we deduce  that
\beg{equation}\label{2.7}\begin{split} \|\nn_{(-A)^\kk\eta}
P_t^0f(x)\|^2_\H &\leq\ff{\E\|f(Z_t^x)\|^2_\H}{t^2}\int_0^t
\|\e^{sA}(-A)^\kk\eta\|^2_\H\d s\\
&\1\ff{P_t^0\|f(x)\|_\H^2}{t^2}
\sum_{j=1}^\8\ff{(1-\e^{-2\ll_jt})\<\eta,e_j\>^2}{\ll_j^{1-2\kk}},~~~~~\kk\in[0,1/2),
\end{split}\end{equation}
which, combining $u^\lambda\in C([0,T]; C_b^1(\mathbb{H};
\mathbb{H}))$ with $\|b\|_{T,\infty}<\8$ and
$\ll_j^{2\kk-1}(1-\e^{-2\ll_jt})\1t^{1-2\kk}$, yields that
\beg{equation*}\begin{split}
\|\nn_{(-A)^\kk\eta}u^\ll_t\|_\H&\leq \int_t^T \e^{-\ll(s-t)}\|\nn_{(-A)^\kk\eta}P_{s-t}^0(\nn_{b_s}u^\ll_s+b_s)\|_\H\d s\\
&\1\|\eta\|_\H\int_{0}^{T} \e^{-\lambda s}s^{-(\kk+\frac{1}{2})}\d
s\\
&\1 \|\eta\|_\H\lambda^{\kk-\frac{1}{2}}.
\end{split}\end{equation*}
Thereby, \eqref{eq8} follows immediately. Next, in \eqref{2.7},
taking $\kk=0$ and $\eta=e_i$ gives that
 \beg{equation}\label{a3}
 \begin{split} \|\nn_{e_i}
P_t^0f(x)\|^2_\H
&\1\ff{(1-\e^{-2\ll_it})P_t^0\|f(x)\|_\H^2}{\ll_it^2}.
\end{split}\end{equation}
This, in addition to $u^\lambda\in C([0,T]; C_b^1(\mathbb{H};
\mathbb{H}))$ and $\|b\|_{T,\infty}<\8$, leads to
\beg{equation}\label{e1}\begin{split}
\|\nn_{e_i}u^\ll_t\|_\H&\leq \int_t^T \e^{-\ll(s-t)}\|\nn_{e_i}P_{s-t}^0(\nn_{b_s}u^\ll_s+b_s)\|_\H\d s\\
&\1\int_0^T \ff{\e^{-\ll
s}(1-\e^{-2\ll_is})^{\ff{1}{2}}}{\ll_i^{\ff{1}{2}}s}\d s\\
&\1\ff{1}{\ll_i^{\ff{1-\theta}{2}}}\int_0^T \e^{-\ll
s}s^{\ff{\theta}{2}-1}\d s\\
&\1\ff{1}{\ll_i^{\ff{1-\theta}{2}}},~~~~~\theta\in(0,1),
\end{split}\end{equation}
where we have used the elementary inequality
\begin{equation}\label{a5}
|\e^{-x}-\e^{-y}|\le
c_\theta|x-y|^\theta,~~~~x,y\ge0,~~\theta\in[0,1]
\end{equation}
for some constant $c_\theta>0.$ Hence, we deduce from \eqref{eq1}
and \eqref{e1} that  \beg{equation*}
 \sum_{i=1}^\8\ll_i^{\aa-\theta}\|\nn_{e_i}u^\ll\|_{T,\8}^2\le
 c_0\,\sum_{i=1}^\8\ff{1}{\ll_i^{1-\aa}}\le c,~~~~\theta\in(0,\aa]
\end{equation*}
for some constants $c_0,c>0.$ As a result, \eqref{w1} holds.
\end{proof}

\begin{rem}
{\rm In fact, all assertions in Lemma \ref{L2.1} hold under the
assumptions ({\bf A1}) and \eqref{e2}. Nevertheless, in Lemma
\ref{L2.1} we write ({\bf A2}) in lieu of \eqref{e2} just to present
in a consistent manner.

}
\end{rem}

The  Lemma   below plays a crucial role in analyzing spatial error
of numerical schemes.

\begin{lem}
{\rm Let  {\bf (A1)} and  ${\bf (A2)}$ hold and assume further
 $\nu:=\ff{\vv+2\bb\wedge\aa\vv^2}{2}+\aa-1>0$. Then, for any $\ll\ge\ll_T,$
\begin{equation}\label{w3}
\sum_{i=1}^\8\ll_i^{\nu}\|\nn_{e_i}\nn_{e_i} u^\ll\|_{T,\8}<\8.
\end{equation}

}
\end{lem}

\begin{proof}
From \eqref{2.6} and $\nn_\eta Z_t^x=\e^{tA}\eta$, we have
\begin{equation}\label{nabla^2}
\ff{1}{2}(\nn_{e_i}\nn_\eta
P_t^0f)(x)=\ff{\e^{-\ff{\ll_it}{2}}}{t}\E\Big((\nn_{e_i}P_{t/2}^0f)(Z^x_{t/2})\int_0^{t/2}\<\e^{sA}\eta,\d
W_s\>\Big).
\end{equation}
By H\"older's inequality and It\^o's isometry, it then follows from
\eqref{a3}, \eqref{a5} with $\theta=1$, contractive property of
$\e^{tA}$ and semigroup property of $P_t^0$ that
\begin{equation}\label{e6}
\begin{split}
 \|(\nn_{e_i}\nn_\eta
P_t^0f)(x)\|_\H^2&\1\ff{\e^{-\ll_it}\|\eta\|_\H^2}{t}\E\|(\nn_{e_i}P_{t/2}^0f)(Z^x_{t/2})\|_\H^2\\
&\1\ff{\e^{-\ll_it}\|\eta\|_\H^2 }{t^2}P_t^0\|f(x)\|_\H^2.
\end{split}
\end{equation}
Furthermore,  It\^o's isometry, \eqref{eq1} as well as \eqref{a5}
with $\theta=\aa$ yield that
\begin{equation}\label{e7}
\E\|Z_t^x-\e^{tA}x\|_\H^2=\int_0^t\|\e^{(t-s)A}\|_{\rm HS}^2\d
s=\sum_{i=1}^\8\ff{1-\e^{-2\ll_it}}{2\ll_i}\1t^\aa.
\end{equation}
For a  mapping $f:\H\rightarrow\H$ such that $\|f(x)-f(y)\|_\H\1
\|x-y\|^\vv_\H$, let $\tt f_t(y)=f(y)-f(\e^{tA}x)$. Taking
\eqref{e6} and \eqref{e7} into account and employing Jensen's
inequality, we derive that
\begin{equation}\label{e8}
\begin{split}
\|(\nn_{e_i}\nn_\eta P_t^0f)(x)\|_\H^2&=\|(\nn_{e_i}\nn_\eta
P_t^0\tt f)(x)\|_\H^2 \\
&\1\ff{\e^{-\ll_it}\|\eta\|_\H^2 \E\| Z^x_t-\e^{tA}x\|^{2\vv}_\H}{t^2}\\
&\1\ff{\e^{-\ll_it}\|\eta\|_\H^2 }{t^{2-\aa\vv}}.
\end{split}
\end{equation}
For notational simplicity, set \begin{equation}\label{w0} \hat
f_t^\ll(x):=(\nn_{b_t}u^\ll_t+b_t)(x).
\end{equation}
It is easy to see from \eqref{e2} and \eqref{u1-0} that
\begin{equation*}
\|\hat f_t^\ll(x)-\hat f_t^\ll(y)\|_\H\1\|x-y\|^\vv_\H,~~~~x,y\in\H.
\end{equation*}
Thus, combining \eqref{PDE} with \eqref{e8}   yields that
\beq\label{e9}\begin{split} \|\nn_{e_i}\nn_\eta
u^\ll_s\|_\H&\le\int_s^T
\e^{-\ll(t-s)}\|\nn_{e_i}\nn_\eta P_{t-s}^{0}\hat f_t^\ll\|_\H\d t\\
&\1\|\eta\|_\H\int_0^T \ff{\e^{-\ll_it/2} }{t^{1-\aa\vv/2}}\d t\\
&\1\ff{\|\eta\|_\H}{\ll_i^{\aa\vv/2}}.
\end{split}
\end{equation}
For a mapping $f:\H\rightarrow\H$ satisfying
\begin{equation}\label{eq0}
\|f(x)-f(y)\|_\H\le
\sum_{i=1}^\8\ll_i^{-\bb_0}|x_i-y_i|^{\vv_0},~~~x,y\in\H
\end{equation}
for some $\bb_0>0$ and $\vv_0\in(0,1)$, let $$\tt
f_t^{i}(y):=f(y)-f(y+\<\e^{tA}x-y,e_i\>e_i),~~~x,y\in\H.$$ Now,
\eqref{eq0}, Jensen's inequality and It\^o's isometry imply that
\begin{equation}\label{a4}
P_t^0\|\tt
f_t^i(x)\|_\H^2\le\ll_i^{-2\bb_0}\E|\Lambda_t^i|^{2\vv_0}\1\ff{(1-\e^{-2\ll_it})^{\vv_0}}{\ll_i^{\vv_0+2\bb_0}},
\end{equation}
where
\begin{equation*}
\Lambda_t^i:=\int_0^t\<\e^{(t-s)A}\d W_s,e_i\>
=\int_0^t\e^{-\ll_i(t-s)}\d\bb_s^{(i)}.
\end{equation*}
In terms of \eqref{e6} with $\eta=e_i$ and \eqref{a4}, besides the
notion of $\tt f_t^i$, it follows from \eqref{a5} that
\begin{equation}\label{a1}
\begin{split}
 \|(\nn_{e_i}\nn_{e_i}
P_t^0f)(x)\|_\H^2&= \|(\nn_{e_i}\nn_{e_i} P_t^0\tt
f_t^i)(x)\|_\H^2\\
&\1\ff{\e^{-\ll_it} P_t^0\|\tt f_t^i(x)\|^2_\H}{t^2}\\
&\1\ff{\e^{-\ll_it}
}{t^{2-\vv_0\theta}\ll_i^{\vv_0(1-\theta)+2\bb_0}},~~~~\theta\in(0,1].
\end{split}
\end{equation}
Next, thanks to \eqref{e3},  $\|b\|_{T,\8}<\8$ and  \eqref{e9}, we
obtain
\begin{equation}\label{a2}
\|\hat f_t^\ll(x)-\hat
f_t^\ll(x+\<y-x,e_i\>e_i)\|_\H\1\ll_i^{-(\bb\wedge\ff{\aa\vv^2}{2})}|x_i-y_i|^\vv,
\end{equation}
 where
$\hat f_t^\ll:\H\rightarrow\H$ is defined in \eqref{w1},
Hence, in light of \eqref{a1} with $\vv_0=\vv$ and
$\bb_0=\bb\wedge\ff{\aa\vv^2}{2}$ and \eqref{a2}, one infers that
\begin{equation*}
\begin{split}
\|\nn_{e_i}\nn_{e_i}
u_s^\ll\|_\H&\1\ff{1}{\ll_i^{\ff{\vv(1-\theta)+2\bb\wedge\aa\vv^2}{2}}}\int_0^T\ff{\e^{-\ff{\ll_it}{2}}
}{t^{1-\ff{\vv\theta}{2}}}\d t\\
&\1\ff{1}{\ll_i^{\ff{\vv+2\bb\wedge\aa\vv^2}{2}}}.
\end{split}
\end{equation*}
This, along with \eqref{eq1}, implies \eqref{w3}.
\end{proof}


The following lemma   provides us with a regular representation of
the continuous-time EI scheme \eqref{eq6}. \beg{lem}\label{L2.2}
{\rm For any $t\in[0,T]$ and $\ll\geq\ll_T$, it holds that
\beg{equation}\begin{split}\label{2.11}
&Y^{(n),\dd}_t+u_t^\ll(Y^{(n),\dd}_t)\\
&=\e^{tA}\{x_n+u_0^\ll(x_n)\}+\int_{0}^{t}\e^{(t-s)A}(\lambda {\rm {\bf I}}-A) u^\lambda_{s}(Y^{(n),\dd}_s)\d s\\
&\quad+\int_{0}^{t}\e ^{(t-s)A}\{\e^{(s-s_\dd)A}b_{s_\dd}^{(n)}(Y^{(n),\dd}_{s_\dd})-b_s(Y^{(n),\dd}_s)\}\d s\\
&\quad+\int_{0}^{t}\e ^{(t-s)A}\nn u_s^\ll(Y^{(n),\dd}_s)\left(\e^{(s-s_\dd)A}b_{s_\dd}^{(n)}(Y^{(n),\dd}_{s_\dd})-b_s(Y^{(n),\dd}_s)\right)\d s\\
&\quad+\frac{1}{2}\sum_{i=1}^n\int_{0}^{t}\e ^{(t-s)A}(\nn^2_{\e^{(s-s_\dd)A}e_i}u^\ll_s -\nn^2_{e_{i}}u^\lambda_s)(Y^{(n),\dd}_s)\d s\\
&\quad+\frac{1}{2}\sum_{i=n+1}^\infty\int_{0}^{t}\e ^{(t-s)A}(\nabla^{2}_{e_{i}}u^\lambda_{s})(Y^{(n),\dd}_s)\d s\\
&\quad+\int_{0}^{t}\e ^{(t-s_\dd)A}\d W^{(n)}_s+\int_{0}^{t}\e
^{(t-s)A}(\nn_{\e^{(s-s_\dd)A}\d W^{(n)}_s} u^\ll_s)(Y^{(n),\dd}_s),
\end{split}\end{equation}
in which ${\bf I}$ is the identity operator on $\H.$ }
\end{lem}
\begin{proof}[Proof]
Since $A_nx=Ax$ for any $x\in\H_n$, \eqref{eq6} can be reformulated
as \beg{equation*}  \d
Y^{(n),\dd}_t=\{AY^{(n),\dd}_t+\e^{(t-t_\dd)A}b_{t_\dd}^{(n)}(Y^{(n),\dd}_{t_\dd})\}\d
t+\e^{(t-t_\dd)A}\d W^{(n)}_t,~~t>0,~~Y^{(n),\dd}_0=x_n.
\end{equation*}
Applying It\^{o}'s formula, for $\ll\ge\ll_T$  we deduce from
\eqref{2.1} that

\beg{align*}  \d
&\{Y^{(n),\dd}_t+u_t^\ll(Y^{(n),\dd}_t)\}\\
&=\Big\{AY^{(n),\dd}_t
+\e^{(t-t_\dd)A}b_{t_\dd}^{(n)}(Y^{(n),\dd}_{t_\dd})+(\partial_tu_t^\ll)(Y^{(n),\dd}_t)+(\nn_{A\cdot}
u_t^\ll)(Y^{(n),\dd}_t)\\
&\quad+\Big(\nn_{\e^{(t-t_\dd)A}b_{t_\dd}^{(n)}(Y^{(n),\dd}_{t_\dd})}
u_t^\ll\Big)(Y^{(n),\dd}_t)+\ff{1}{2}\sum_{k=1}^n\Big(\nn^2_{\e^{(t-t_\dd)A}e_k}u_t^\ll\Big)(Y^{(n),\dd}_t)\Big\}\d t\\
&\quad+\e^{(t-t_\dd)A}\d W^{(n)}_t+\Big(\nn_{\e^{(t-t_\dd)A}\d
W^{(n)}_t}
u_t^\ll\Big)(Y^{(n),\dd}_t)\\
&=\Big\{AY^{(n),\dd}_t
+\ll u_t^\ll(Y^{(n),\dd}_t)+\e^{(t-t_\dd)A}b_{t_\dd}^{(n)}(Y^{(n),\dd}_{t_\dd})-b_t(Y^{(n),\dd}_t)\\
&\quad+\nn
u_t^\ll(Y^{(n),\dd}_t)\left(\e^{(t-t_\dd)A}b_{t_\dd}^{(n)}(Y^{(n),\dd}_{t_\dd})-b_t(Y^{(n),\dd}_t)\right)+\ff{1}{2}\sum_{k=1}^n\Big(\nn^2_{\e^{(t-t_\dd)A}e_k}u_t^\ll\Big)(Y^{(n),\dd}_t)\\
&\quad-\ff{1}{2}\sum_{k=1}^\infty(\nn^2_{e_k}u_t^\ll)(Y^{(n),\dd}_t)\Big\}\d
t+\e^{(t-t_\dd)A}\d W^{(n)}_t+\Big(\nn_{\e^{(t-t_\dd)A}\d W^{(n)}_t}
u_t^\ll\Big)(Y^{(n),\dd}_t).
\end{align*}
This, in addition to integration by parts, further implies that
\begin{align*}
&\int_0^tA\e^{(t-s)A}\{Y^{(n),\dd}_s+u_s^\ll(Y^{(n),\dd}_s)\}\d s\\
&=-\e^{(t-s)A}\{Y^{(n),\dd}_s+u_s^\ll(Y^{(n),\dd}_s)\}\Big|_0^t+\int_0^t\e^{(t-s)A}\d\{Y^{(n),\dd}_s+u_s^\ll(Y^{(n),\dd}_s)\}\\
&=-\{Y^{(n),\dd}_t+u_t^\ll(Y^{(n),\dd}_t)\}+\e^{tA}\{x_n+u_0^\ll(x_n)\}+\int_0^tA\e^{(t-r)A}Y^{(n),\dd}_r
\d
r\\
&\quad+\ll\int_0^t\e^{(t-r)A} u_r^\ll(Y^{(n),\dd}_r)\d
r+\int_0^t\e^{(t-r)A}\{\e^{(r-r_\dd)A}b_{r_\dd}^{(n)}(Y^{(n),\dd}_{r_\dd})-b_r(Y^{(n),\dd}_r)\}\d
r\\
&\quad+\int_0^t\e^{(t-r)A}\nn u_r^\ll(Y^{(n),\dd}_r)\left(\e^{(r-r_\dd)A}b_{r_\dd}^{(n)}(Y^{(n),\dd}_{r_\dd})-b_r(Y^{(n),\dd}_r)\right)\d r\\
&\quad+\ff{1}{2}\sum_{k=1}^n\int_0^t\e^{(t-r)A}\Big(\nn^2_{\e^{(r-r_\dd)A}e_k}u_r^\ll\Big)(Y^{(n),\dd}_r)\d
r-\ff{1}{2}\sum_{k=1}^\infty\int_0^t\e^{(t-r)A}(\nn^2_{e_k}u_r^\ll)(Y^{(n),\dd}_r)\d
r\\
&\quad+\int_0^t\e^{(t-r_\dd)A}\d
W^{(n)}_r+\int_0^t\e^{(t-r)A}\Big(\nn_{\e^{(r-r_\dd)A}\d W^{(n)}_r}
u_r^\ll\Big)(Y^{(n),\dd}_r).
\end{align*}
As a consequence, the desired assertion \eqref{2.11} is now
available.
\end{proof}


The following lemma   concerns the continuity in the mean $L^2$-norm
sense for the displacement of $(Y^{(n),\dd}_t)_{t\in[0,T]}$.

\beg{lem}\label{L2.4}{\rm  Let {\bf (A1)} and  {\bf (A2)} hold and
assume  that  the inial value $x\in\D(A)$. Then,
 \beg{equation}\label{q1}
\sup_{t\in[0,T]}\mathbb{E}\|Y^{(n),\dd}_t-Y^{(n),\dd}_{t_\dd}\|_\H^2\1_T\dd^\aa.
\end{equation}
}
\end{lem}
\beg{proof} To make the content self-contained, we here give a
sketch although the corresponding argument of \eqref{q1} is quite
standard. By virtue of \eqref{eq6}, it follows immediately that
\begin{align*}
Y^{(n),\dd}_t-Y^{(n),\dd}_{t_\dd}
&=(\e^{(t-t_\dd)A}-{\rm {\bf I}})\e^{{t_\dd}A}x_n\\
&\quad+\int_0^{t_\dd}(\e^{(t-t_\dd)A}-{\rm {\bf
I}})\e^{(t_\dd-s_\dd)A}b_{s_\dd}^{(n)}(Y^{(n),\dd}_{s_\dd})\d
s\\
&\quad+\int_0^{t_\dd}(\e^{(t-t_\dd)A}-{\rm {\bf I}})\e^{(t_\dd-s_\dd)A}\d W^{(n)}_s\\
&\quad+\int_{t_\dd}^t\e^{(t-s_\dd)A}b_{s_\dd}^{(n)}(Y^{(n),\dd}_{s_\dd})\d
s+\int_{t_\dd}^t\e^{(t-s_\dd)A}\d W^{(n)}_s.
\end{align*}
Recall the elementary inequalities: for any $\gamma\in(0,1],$
\begin{equation}\label{eq5}
\|(-A)^\gamma\e^{tA}\|\le
t^{-\gamma}~~~~~\mbox{and}~~~~~~\|(-A)^{-\gamma}(\e^{tA}-{\rm {\bf
I}})\|\le t^\gamma.
\end{equation}
Next, according to  H\"older's inequality and  It\^o's isometry and
by taking contractive property of $\e^{tA}$ and $\|b\|_{T,\8}<\8$
into account, we derive from \eqref{eq5} that
\begin{align*}
\E\|Y^{(n),\dd}_t-Y^{(n),\dd}_{t_\dd}\|_\H^2&\1_T\|(\e^{(t-t_\dd)A}-{\rm {\bf I}})\e^{{t_\dd}A}x_n\|_\H^2\\
&\quad+\int_0^{t_\dd}\E\|(\e^{(t-t_\dd)A}-{\rm {\bf
I}})\e^{(t_\dd-s_\dd)A}b_{s_\dd}^{(n)}(Y^{(n),\dd}_{s_\dd})\|_\H^2\d
s\\
&\quad+\int_0^{t_\dd}\|(\e^{(t-t_\dd)A}-{\rm {\bf I}})\e^{(t_\dd-s_\dd)A}\|_{\mathscr{L}_2}^2\d  s\\
&\quad+\dd\int_{t_\dd}^t\E\|\e^{(t-s_\dd)A}b_{s_\dd}^{(n)}(Y^{(n),\dd}_{s_\dd})\|_\H^2\d
s+\int_{t_\dd}^t\|\e^{(t-s_\dd)A}\|_{\mathscr{L}_2}^2\d s\\
&\1_T\|(\e^{(t-t_\dd)A}-{\rm {\bf I}})(-A)^{-1}\|^2\|A x\|_\H^2\\
&\quad+\int_0^{t_\dd}\|(\e^{(t-t_\dd)A}-{\rm {\bf
I}})(-A)^{-\aa/2}\|^2\|(-A)^{\aa/2}\e^{(t_\dd-s_\dd)A}\|^2\E\|b_{s_\dd}^{(n)}(Y^{(n),\dd}_{s_\dd})\|_\H^2\d
s\\
&\quad+\int_0^{t_\dd}\|(\e^{(t-t_\dd)A}-{\rm {\bf I}})(-A)^{-\aa/2}\|^2\|(-A)^{\aa/2}\e^{(t_\dd-s_\dd)A}\|_{\mathscr{L}_2}^2\d  s\\
&\quad+\dd\int_{t_\dd}^t\E\|b_{s_\dd}^{(n)}(Y^{(n),\dd}_{s_\dd})\|_\H^2\d
s+\int_{t_\dd}^t\|\e^{(t-s_\dd)A}\|_{\mathscr{L}_2}^2\d s\\
&\1_T\dd^2+\dd^{\aa}\int_0^{t_\dd}(t_\dd-s_\dd)^{-\aa}\d
s+\dd^{\aa}\sum_{i=1}^\8\ll_i^{\aa}\int_0^{t_\dd}\e^{-2\ll_i(t_\dd-s_\dd)}\d
s\\
&\quad+\sum_{i=1}^\8\int_0^{t-t_\dd}\e^{-2\ll_is}\d
s\\
&\1_T\dd^2+\dd^{\aa}\int_0^{t_\dd}s^{-\aa}\d
s+\dd^{\aa}\sum_{i=1}^\8\ll_i^{\aa}\int_0^{t_\dd}\e^{-2\ll_is}\d
s+\sum_{i=1}^\8\int_0^{t-t_\dd}\e^{-2\ll_is}\d
s\\
&\1_T\dd^{\aa}+\dd^{\aa}\sum_{i=1}^\8\ff{1}{\ll_i^{1-\aa}}+\sum_{i=1}^\8\ff{1-\e^{-2\ll_i(t-t_\dd)}}{\ll_i}\\
&\1_T\dd^{\aa}+\dd^{\aa}\sum_{i=1}^\8\ff{1}{\ll_i^{1-\aa}}\\
&\1_T\dd^{\aa},
\end{align*}
where in the penultimate display we have used \eqref{a5} with
$\theta=\aa$ and in the last step utilized \eqref{eq1}.
\end{proof}

\section{Proof of Theorem \ref{T1.1}}\label{sec3}
With Lemmas \ref{L2.1}-\ref{L2.4} in hand, we now in a position to
complete the proof of Theorem \ref{T1.1}. In view  of \eqref{u1-0},
we deduce that there exists $\hat\ll_T\ge\ll_T$ such that  for any
$\ll\ge\hat\ll_T$
\begin{equation}\label{eq7}
\|u^\ll\|_{T,\8}+\|\nn u^\ll\|_{T,\infty}+\|\nn u^\ll
(-A)^\kk\|_{T,\infty}+\|\nabla^{2} u^\lambda
\|_{T,\infty}\le\ff{1}{3\ss2}.
\end{equation}
In what follows, we shall fix $\ll\ge\hat\ll_T$ so that \eqref{eq7}
holds. For notational simplicity, set
$\theta_t^\ll(x):=x+u_t^\ll(x), x\in\H.$ Since
\begin{equation*}
\begin{split}
\|X_t-Y^{(n),\dd}_t\|_\H^2&\le2\|\theta_t^\ll(X_t)-\theta_t^\ll(Y^{(n),\dd}_t)\|_\H^2+2\|u_t^\ll(X_t)-u_t^\ll(Y^{(n),\dd}_t)\|_\H^2\\
&\le2\|\theta_t^\ll(X_t)-\theta_t^\ll(Y^{(n),\dd}_t)\|_\H^2+\ff{1}{9}\|X_t-Y^{(n),\dd}_t\|_\H^2,
\end{split}
\end{equation*}
we have
\begin{equation}\label{eq4}
\GG_t^{(n),\dd}:=\E\|X_t-Y^{(n),\dd}_t\|_\H^2\le\ff{9}{4}\E\|\theta_t^\ll(X_t)-\theta_t^\ll(Y^{(n),\dd}_t)\|_\H^2.
\end{equation}
According to \cite[Proposition 2.5]{W}, one has
\begin{equation}\label{eq3}
\begin{split}
\theta_t^\ll(X_t)&=\e^{tA}\theta_0^\ll(x)+\int_0^t\{(\ll{\bf
I}-A)\e^{(t-s)A}u_s^\ll(X_s)\}\d s\\
&\quad+\int_0^t\e^{(t-s)A}\{\d W_s+(\nn_{\d W_s} u^\ll_s)(X_s)\}.
\end{split}
\end{equation}
In view of \eqref{2.11}  and \eqref{eq3}, we find that
\begin{equation} \label{3.3}
\begin{split}
\GG_t^{(n),\dd}
&\le 9\Big\{\|\e^{tA}(x-x_n)\|^2_\H+\|\e^{tA}(u_0^\ll(x)-u_0^\ll(x_n))\|^2_\H\\
&\quad+\E\Big\|\int_{0}^{t}\e^{(t-s)A}(\ll{\bf
I}-A)\{u^\lambda_{s}(X_{s})-u^\ll_{s}(Y^{(n),\dd}_s)\}\d s\Big\|^2_\H\\
&\quad+\E\Big\|\int_0^t\e ^{(t-s)A}\{\e^{(s-s_\dd)A}b_{s_\dd}^{(n)}(Y^{(n),\dd}_{s_\dd})-b_{s}(Y^{(n),\dd}_s)\}\d s\Big\|^2_\H\\
&\quad+\E\Big\|\int_{0}^{t}\e ^{(t-s)A}\nn u_s^\ll(Y^{(n),\dd}_s)\left(\e^{(s-s_\dd)A}b_{s_\dd}^{(n)}(Y^{(n),\dd}_{s_\dd})-b_s(Y^{(n),\dd}_s)\right)\d s\Big\|^2_\H\\
&\quad+\E\Big\|\sum_{i=1}^n\int_0^t\e ^{(t-s)A} \Big(\nabla^{2}_{\e^{(s-s_\dd)A}e_i}u^\lambda_{s} -\nabla^{2}_{e_{i}}u^\lambda_{s}\Big)(Y^{(n),\dd}_s)\d s\Big\|^2_\H\\
&\quad+\E\Big\|\sum_{i=n+1}^\8\int_0^t\e
^{(t-s)A}\Big(\nn^2_{e_{i}}u^\ll_s\Big)(Y^{(n),\dd}_s)\d
s\Big\|^2_\H\\
&\quad+\E\Big\|\int_{0}^{t}\e^{(t-s)A}\d W_s-\int_{0}^{t}\e ^{(t-s_\dd)A}\d W^{(n)}_s\Big\|^2_\H\\
&\quad+\E\Big\|\int_{0}^{t}\e^{(t-s)A}\Big(\nn_{\d W_s} u^\ll_s\Big)(X_{s})-\int_{0}^{t}\e ^{(t-s)A}\Big(\nn_{\e^{(s-s_\dd)A}\d W^{(n)}_s} u^\ll_s\Big)(Y^{(n),\dd}_s)\Big\|^2_\H\Big\}\\
&=:9(\Lambda_t^{(1)}+\Lambda_t^{(2)}+\cdots+\Lambda_t^{(9)}).
\end{split}
 \end{equation}
Using H\"older's inequality and Fubini's theorem, we deduce from
\eqref{eq7} that
\begin{align*}
\int_0^t\e^{-2\ll s}\Lambda_s^{(3)}\d s
&=\sum_{i=1}^\8\E\int_0^t\e^{-2\ll
s}\Big(\int_0^s\<\e^{(s-r)A}(\ll{\bf
I}-A)\{u^\ll_r(X_r)-u^\ll_r(Y^{(n),\dd}_r)\},e_i\>\d r\Big)^2 \d s\\
&=\sum_{i=1}^\8\E\int_0^t\Big(\int_0^s\e^{-\ll_i(s-r)-\ll
s}(\ll+\ll_i)\<u^\ll_r(X_r)-u^\ll_r(Y^{(n),\dd}_r),e_i\>\d r\Big)^2
\d s\\
&\le\sum_{i=1}^\8\int_0^t\Big(\int_0^s\e^{-(\ll_i+\ll)
(s-r)}(\ll+\ll_i)\d r
\\
&\quad\times\int_0^s\e^{-(\ll+\ll_i)(s-r)-2\ll
r}(\ll+\ll_i)\E\<u^\ll_r(X_r)-u^\ll_r(Y^{(n),\dd}_r),e_i\>^2\d
r\Big) \d
s\\
&\le\sum_{i=1}^\8\int_0^t\int_0^s\e^{-(\ll+\ll_i)(s-r)-2\ll
r}(\ll+\ll_i)\E\<u^\ll_r(X_r)-u^\ll_r(Y^{(n),\dd}_r),e_i\>^2\d r \d
s\\
&=\sum_{i=1}^\8\int_0^t\Big(\int_r^t\e^{-(\ll+\ll_i)(s-r)}(\ll+\ll_i)\d
s\Big)\e^{-2\ll r}\E\<u^\ll_r(X_r)-u^\ll_r(Y^{(n),\dd}_r),e_i\>^2 \d
r\\
&\le\sum_{i=1}^\8\int_0^t\e^{-2\ll
s}\E\<u^\ll_s(X_s)-u^\ll_s(Y^{(n),\dd}_s),e_i\>^2 \d s\\
&=\int_0^t\e^{-2\ll
s}\E\|u^\ll_s(X_s)-u^\ll_s(Y^{(n),\dd}_s)\|_\H^2 \d s\\
&\le\ff{1}{18}\int_0^t\e^{-2\ll s}\GG_s^{(n),\dd} \d s,
\end{align*}
which, combining with \eqref{3.3}, further leads to
\begin{equation*}
\int_0^t\e^{-2\ll s}\GG_s^{(n),\dd} \d s\le
\ff{1}{2}\int_0^t\e^{-2\ll s}\GG_s^{(n),\dd}\d
s+9\sum_{i=1,i\neq3}^9\int_0^t\e^{-2\ll s}\Lambda_s^{(i)}\d s.
\end{equation*}
We therefore obtain that
\begin{equation}\label{Gamma^n}
\int_0^t\e^{-2\ll s}\GG_s^{(n),\dd} \d s\le 18\sum_{i=1,i\neq3}^9\int_0^t\e^{-2\ll s}\Lambda_s^{(i)}\d s.
\end{equation}
 To achieve the desired assertion \eqref{eq12}, in the sequel, we aim
to estimate the terms $\Lambda_t^{(i)}, i\neq3,$ step-by-step. By
the contraction property of $\e^{tA}$  and thanks to \eqref{eq7} and
$x\in\D(A)$,
\begin{equation*}
\Lambda_t^{(1)}+\Lambda_t^{(2)}\1\|x-x_n\|_\H^2
\1\ff{1}{\ll_n^2}\sum_{i=n+1}^\8\ll_i^2\<x,e_i\>^2\1\ff{\|Ax\|_\H^2}{\ll_n^2}\1\ff{1}{\ll_n^2}.
\end{equation*}
Taking  H\"older's inequality and \eqref{eq5} with $\gg=\aa\vv/2$
into consideration and taking advantage of  contractive property of
$\e^{tA}$ and $\|b\|_{T,\8}<\8$ as well as {\bf(A3)}, we deduce from
$\aa\vv\in(0,1)$ that
\begin{equation*}
\begin{split}
\Lambda_t^{(4)}
&\1_T\int_0^t\E\|b_{s_\dd}(Y^{(n),\dd}_{s_\dd})-b_s(Y^{(n),\dd}_{s_\dd})\|_\H^2\d
s+\int_0^t\E\|b_s(Y^{(n),\dd}_{s_\dd})-b_s(Y^{(n),\dd}_s)\|_\H^2\d
s\\
&\quad+\int_0^t\E\|\e ^{(t-s_\dd)A}\{\pi_n-{\bf
I}\}b_s(Y^{(n),\dd}_s)\|_\H^2\d s+\int_0^t\E\|\e
^{(t-s)A}\{\e^{(s-s_\dd)A}-{\bf
I}\}b_s(Y^{(n),\dd}_s)\|_\H^2\d s\\
&\1_T\dd^{\aa\vv}+\sum_{i=n+1}^\8\int_0^t\e^{-2\ll_i(t-s_\dd)}\E\<b_s(Y^{(n),\dd}_s),e_i\>^2\d
s\\
&\quad+\int_0^t\|\e
^{(t-s)A}(-A)^{\aa\vv/2}\|^2\|(-A)^{-\aa\vv/2}\{\e^{(s-s_\dd)A}-{\bf
I}\}\|^2\E\|b_s(Y^{(n),\dd}_s)\|_\H^2\d s\\
&\1_T\dd^{\aa\vv}+\int_0^t\e^{-2\ll_n(t-s_\dd)}\E\|b_s(Y^{(n),\dd}_s)\|_\H^2\d
s+\dd^{\aa\vv}\int_0^t(t-s)^{-\aa\vv}\E\|b_s(Y^{(n),\dd}_s)\|_\H^2\d s\\
&\1_T\dd^{\aa\vv}+\ff{1}{\ll_n}.
\end{split}
\end{equation*}
  By the aid  of Jensen's inequality, in
addition to \eqref{e2}, \eqref{u1-0}, \eqref{q1}, \eqref{eq5},
\eqref{eq7}, contractive property of $\e^{tA}$ and $\|b\|_{T,\8}<\8$
along with {\bf(A3)} yield that
\begin{align*}
\Lambda_t^{(5)}
&\1\int_{0}^{t}\E\|b_{s_\dd}(Y^{(n),\dd}_{s_\dd})-b_{s_\dd}(Y^{(n),\dd}_s)\|^2_\H\d
s\\
&\quad+\int_{0}^{t}\E\|b_{s_\dd}(Y^{(n),\dd}_s)-b_s(Y^{(n),\dd}_s)\|^2_\H\d
s\\
&\quad+\int_{0}^{t}\E\|\nn
u_s^\ll(Y^{(n),\dd}_s)(-A)^{\aa\vv/2}\|^2\|(-A)^{-\aa\vv/2}(\e^{(s-s_\dd)A}-{\bf
I})\|^2\d
s\\
&\quad+\int_{0}^{t}\E\|\nn
u_s^\ll(Y^{(n),\dd}_s)(-A)^{\nu}\|^2\|(-A)^{-\nu}(\pi_n-{\bf
I})\|^2\d s\\
&\1\dd^{\aa\vv} +\ff{1}{\ll_n^{2\nu}},
\end{align*}
where we have also used
$\ff{1}{2}>\nu=\ff{\vv+2\bb\wedge\aa\vv^2}{2}+\aa-1>0$.
We thus obtain from
 \eqref{w3} that
\begin{equation}\label{eq9}
\begin{split}
\Lambda_t^{(6)}&\1\int_0^t\E\Big(\sum_{i=1}^n(1-\e^{-\ll_i(s-s_\dd)})^2
\|(\nabla_{e_i}^2u^\lambda_s
)(Y^{(n),\dd}_s)\|_\H\Big)^2\d s\\
&\1\dd^{2\nu}\int_0^t\E\Big(\sum_{i=1}^n\ll_i^{\nu}
\|(\nabla_{e_i}^2u^\lambda_s
)(Y^{(n),\dd}_s)\|_\H\Big)^2\d s\\
&\1\dd^{2\nu},
\end{split}
\end{equation}
where we have used  $\sup_{x>0}\{(1-\e^{-x})x^{-\nu}\}<\8$ for
$\nu\in(0,1)$. Also, with the help of \eqref{w3},
\begin{equation}\label{eq10}
\begin{split}
\Lambda_t^{(7)}&\1\int_0^t\E\Big(\sum_{i=n+1}^\infty\|(\nabla^{2}_{e_{i}}u^\lambda_{s})(Y^{(n),\dd}_s)\|_\H\Big)^2\d
s\\
&\1\ff{1}{\ll_n^{2\nu}}\int_0^t\E\Big(\sum_{i=n+1}^\infty\ll_i^{\nu}\|(\nabla^{2}_{e_{i}}u^\lambda_{s})(Y^{(n),\dd}_s)\|_\H\Big)^2\d
s\\
&\1\ff{1}{\ll_n^{2\nu}}.
\end{split}
\end{equation}
Employing  It\^o's isometry, \eqref{eq1} and \eqref{a5} with
$\theta=\aa/2$, we get that
\begin{equation*}
\begin{split}
\Lambda_t^{(8)}&\1\sum_{i=n+1}^\8 \int_{0}^{t}\e^{-2\ll_i(t-s)}\d
s+\sum_{i=1}^n\int_0^t\e^{-2\ll_i(t-s)}(1-\e^{-\ll_i(s-s_\dd)})^2\d s\\
&\1\sum_{i=n+1}^\8\int_0^t\e^{-2\ll_is}\d s+\dd^\aa\sum_{i=1}^n\ll_i^\aa\int_0^t\e^{-2\ll_is}\d s\\
&\1\sum_{i=n+1}^\8\ff{1}{\ll_i}+\dd^\aa\sum_{i=1}^n\ff{1}{\ll_i^{1-\aa}}\\
&\1\ff{1}{\ll_n^\aa}+\dd^\aa.
\end{split}
\end{equation*}
Let
\begin{equation*}
\Theta_t=\int_0^t\|\e^{(t-r)A}\|^2_{\mathrm{HS}}\E\|(\nn
u^\ll_r)(X_r)-(\nn u^\ll_r)(Y^{(n),\dd}_r)\|^2\d r.
\end{equation*}
Since $0<\nu<\ff{1}{2}\wedge\aa$, again, an application of It\^o's
isometry, together with  \eqref{a5} with $\theta=\nu/2$, implies
that
\begin{align*}
\Lambda_t^{(9)} &\le c_1\Big\{\Theta_t+\sum_{i=1}^n
\int_0^t(1-\e^{-\ll_i(r-r_\dd)})^2\E\|(\nn_{e_i}
u^\ll_r)(Y^{(n),\dd}_r)\|^2_\H\d r\\
&\quad+\sum_{i=n+1}^\8\int_0^t\E\|(\nn_{e_i}
u^\ll_r)(X_r)\|_\H^2\d r\Big\}\\
&\le c_2\Big\{\Theta_t+\dd^{\nu}\sum_{i=1}^n\ll_i^{\nu}
\int_0^t\E\|(\nn_{e_i}
u^\ll_r)(Y^{(n),\dd}_r)\|^2_\H\d r\\
&\quad+\ll_n^{-\nu}\sum_{i=n+1}^\8\ll_i^{\nu}\int_0^t\E\|(\nn_{e_i}
u^\ll_r)(X_r)\|_\H^2\d r\Big\}\\
&\le c_3\{\Theta_t+\dd^{\nu}+\ll_n^{-\nu}\}
\end{align*}
for some constants $c_1,c_2,c_3>0$, where we have used    \eqref{w1}
with $\theta=\nu$ in the last procedure. By means of H\"older's
inequality, \eqref{eq1} and \eqref{a5} with $\theta=\ff{\aa}{2}$, we
find that
\begin{equation}\label{w2}
\begin{split}
\int_0^t\e^{-2\lambda s}\|\e^{sA}\|^2_{\mathrm{HS}}\d
s&=\sum_{i=1}^\8\int_0^t\e^{-2\lambda s}\e^{-2\ll_is} \d
s\\
&\le\sum_{i=1}^\8\Big(\int_0^t\e^{-\ff{(2-\aa)\ll_is}{1-\aa}} \d
s\Big)^{\ff{2(1-\aa)}{2-\aa}}\Big(\int_0^t\e^{-\ff{2(2-\aa)\ll s}{\aa}}\d s\Big)^{\ff{\aa}{2-\aa}}\\
&=\sum_{i=1}^\8\Big(\ff{1-\e^{-\ff{(2-\aa)\ll_it}{1-\aa}}}{\ff{(2-\aa)\ll_i}{1-\aa}}\Big)^{\ff{2(1-\aa)}{2-\aa}}\Big(\ff{\aa}{2(2-\aa)\ll}\Big)^{\ff{\aa}{2-\aa}}\\
&\le   c_4\,\ll^{-\ff{\aa}{2-\aa}}
\end{split}
\end{equation}
for some constant $  c_4>0$. Next, via Fubini's theorem, \eqref{eq7}
and \eqref{w2}, we deduce that
\begin{equation*}
\begin{split}
\int_0^t\e^{-2\lambda s}\Theta_s\d s &=\int_0^t\e^{-2\lambda
r}\E\|(\nn u^\ll_r)(X_r)-(\nn
u^\ll_r)(Y^{(n),\dd}_r)\|^2\left(\int_0^{t-r}\e^{-2\lambda s}\|\e^{sA}\|^2_{\mathrm{HS}}\d s\right)\d r\\
&\le c_4\,\ll^{-\ff{\aa}{2-\aa}}\int_0^t\e^{-2\lambda r}\E\|(\nn
u^\ll_r)(X_r)-(\nn
u^\ll_r)(Y^{(n),\dd}_r)\|^2\d r\\
& \le \ff{c_4}{18\ll^{\ff{\aa}{2-\aa}} }\int_0^t\e^{-2\lambda
s}\Gamma_s^{(n),\dd}\d s.
\end{split}
\end{equation*}
Furthermore, taking $\ll\ge\hat\ll_T$ such that $
\ff{c_3c_4}{\ll^{\ff{\aa}{2-\aa}} }<1$ and
combining all the  estimates above with \eqref{Gamma^n}, we arrive
at
\begin{equation*}\begin{split}
\int_0^t\e^{-2\lambda s}\Gamma_s^{(n),\dd}\d s
&\1_T\dd^{\nu}+\ff{1}{\ll_n^{\nu}},
\end{split}
\end{equation*}
where we have also utilized $\nu\in(0,\aa)$ and $\nu\in(0,\aa\vv)$.
This  therefore implies the desired assertion.

\paragraph{Acknowledgement.} The author would like to thank Professor Feng-Yu Wang for corrections and helpful comments.

\beg{thebibliography}{99}

{\small

\setlength{\baselineskip}{0.14in}
\parskip=0pt

\bibitem{BHY}Bao, J., Huang, X., Yuan, C., Convergence Rate of Euler-Maruyama Scheme for SDEs with Rough
Coefficients,  arXiv:1609.06080.

\bibitem{BL}
 Barth, A., Lang, A., $L^p$ and almost sure convergence of a Milstein scheme for stochastic partial differential equations,
  {\it Stochastic Process. Appl.}, {\bf123} (2013), 1563--1587.




\bibitem{D11} Debussche, A., Weak approximation of stochastic partial
differential equations: the nonlinear case, {\it Math. Comp.},
{\bf273} (2011), 89--117.

\bibitem{GZ}  Da Prato, G., Zabczyk, J., \emph{Stochastic Equations in Infinite Dimensions,} Cambridge University Press, Cambridge, 1992.

\bibitem{DJM} Da Prato,  G.,  Jentzen, A.,    R\"ockner, M, A mild It\^o formula for
SPDEs, arXiv:1009.3526v4.








\bibitem{GS}  Gy\"ongy, I.,   Sabanis, S.,   \u{S}i\u{s}ka, Convergence of tamed Euler schemes for a class
of stochastic evolution equations, {\it Stoch PDE: Anal. Comp.},
{\bf 4} (2016), 225--245.

\bibitem{G03}Gy\"ongy, I.,   Krylov, N., On the splitting-up method and
stochastic partial differential equations, {\it Ann. Probab.},
{\bf31} (2003), 564--591.

\bibitem{H03}Hausenblas, E., Approximation for semilinear stochastic evolution
equations, {\it Potential Anal.} {\bf 18} (2003), 141--186.

\bibitem{JP} Jentzen, A., Pu\u{s}nik, P., Strong convergence rates for an explicit numerical approximation method for stochastic evolution equations with
non-globally Lipschitz continuous nonlinearities, arXiv:1504.03523.

\bibitem{JR15} Jentzen, A.,    R\"ockner, M., A Milstein Scheme for
SPDEs, {\it Found Comput Math}, {\bf15} (2015), 313--362.

\bibitem{J11}  Jentzen, A., Kloeden, P.~E., Winkel, G., Efficient simulation of nonlinear parabolic SPDEs with additive
noise,
 {\it Ann. Appl. Probab.}, {\bf21} (2011), 908--950.

\bibitem{JK11}Jentzen, A., Kloeden, P.~ E., Taylor approximations for
stochastic partial differential equations,   Philadelphia, PA, 2011.


\bibitem{JK} Jentzen, A., Kloeden, P.~E., Overcoming the order barrier in the numerical approximation of stochastic partial
 differential equations with additive space-time noise,
 {\it Proc. R. Soc. A}, {\bf465} (2009),  649--667.

\bibitem{KLNS} Kloeden, P.~E.,   Lord, G.~J., Neuenkirch,  A.,  Shardlow, T.,
 The exponential integrator scheme for stochastic partial
differential equations: Pathwise error bounds, {\it J. Comput. Appl.
Math.}, {\bf235} (2011),  1245--1260.

\bibitem{K}
 Kruse, R., Optimal error estimates of Galerkin finite element methods for stochastic partial differential equations with multiplicative
 noise,
  {\it IMA J. Numer. Anal.}, {\bf34} (2014),  217--251.

\bibitem{LCP}
Lang, A., Chow, P.- L.,  Potthoff, J., Almost sure convergence for a
semidiscrete Milstein scheme for SPDEs of Zakai type, {\it
Stochastics}, {\bf82}  (2010), 315--326.

\bibitem{GR}  Lord, G.~J.,  Rougemont, J., A Numerical Scheme for Stochastic PDEs with Gevrey Regularity, {\it IMA J. Numer. Anal.},
{\bf 24} (2004), 587--604.

\bibitem{GT} Lord,  G.~J.,    Shardlow, T.,   Postprocessing for stochastic parabolic partial differential equations,
{\it SIAM J. Numer. Anal.},{\bf 45} (2007), 870--899.

\bibitem{S99}Shardlow, T., Numerical methods for stochastic parabolic PDEs, {\it Numer. Funct. Anal. Optim.}, {\bf20}  (1999), 121--145.

\bibitem{W} Wang, F.-Y.,
 Gradient estimate and applications for SDEs in Hilbert space with
multiplicative noise and Dini continuous drift,  {\it J.
Differential Equations}, {\bf260} (2016),  2792--2829.

}
\end{thebibliography}

\end{document}